\documentclass[11pt]{amsart}
\usepackage{amsmath,amsfonts,amscd,amssymb,epsfig,euscript,bm}
\usepackage{amsxtra,mathrsfs,empheq}
\usepackage{pdfsync}
\newtheorem{proposition}{Proposition}
\newtheorem{lemma}{Lemma}

\theoremstyle{definition}
{}

\theoremstyle{remark} 
\newtheorem{remark}{Remark}

\newcommand{\field}[1]{\ensuremath{\mathbb{#1}}}
\newcommand{\CC}{\field{C}}

\DeclareMathOperator{\Tr}{tr} 

\newcommand{\del}{\partial}
\newcommand{\delb}{\bar\partial}

\newcommand{\C}{{\mathbb{C}}}

\newcommand{\ch}{\mathrm{ch}}
\newcommand{\ta}{\theta}
\newcommand{\tam}{\bm\theta}
\newcommand{\Tam}{\bm\Theta}
\newcommand{\tab}{\bar{\theta}}

\usepackage{hyperref}
\hypersetup{colorlinks,citecolor=blue,plainpages=false,hypertexnames=false}
\begin{document}
\title[Computation of Chern forms]{Explicit computation of the Chern character forms}
\author{Leon A. Takhtajan}
\address{Department of Mathematics \\
Stony Brook University\\ Stony Brook, NY 11794-3651 \\ USA; Euler Mathematical Institute, Saint Petersburg, Russia}
\date{}
\maketitle
\begin{abstract} We propose a method for explicit computation of the Chern character form of a holomorphic Hermitian vector bundle $(E,h)$ over a complex manifold $X$ in a local holomorphic frame. First, we use  the descent equations arising in the double complex of $(p,q)$-forms on $X$ and find explicit degree decomposition of the Chern-Simons form $\mathrm{cs}_{k}$ associated to the Chern character form $\ch_{k}$ of $(E,h)$. Second, we introduce the `ascent' equations that start from the $(2k-1,0)$ component of  $\mathrm{cs}_{k}$, and use Cholesky decomposition of the Hermitian metric $h$ to represent the Chern-Simons form, modulo $d$-exact forms, as a $\del$-exact form. This yields a formula for the Bott-Chern form $\mathrm{bc}_{k}$ of type $(k-1,k-1)$ such that $\displaystyle{\ch_{k}=\frac{\sqrt{-1}}{2\pi}\delb\del\mathrm{bc}_{k}}$.  Explicit computation is presented for the cases $k=2$ and $3$.
\end{abstract}

\section{Introduction}
Let $V$ be a $C^{\infty}$-complex vector bundle with a connection $\nabla=d+A$ over a smooth manifold $X$.  The Chern character form $\ch(V,\nabla)$ for the pair $(V,\nabla)$ is defined by
$$\ch(V,\nabla)=\Tr\left\{\exp\left(\frac{\sqrt{-1}}{2\pi}\,\nabla^{2}\right)\right\}.$$
Here $\nabla^{2}$ is the curvature of the connection $\nabla$, $\mathrm{End} \,V$-valued $2$-form on $X$, 
and $\Tr$ is the trace in the endomorphism bundle $\mathrm{End} \,V$. The Chern character form is closed, $d\,\ch(V,\nabla)=0$, and its cohomology class in $H^{\ast}(X,\C)$ does not depend on the choice of  $\nabla$ (see, e.g., \cite{GH}).

Let $\nabla^{0}$ and $\nabla^{1}$ be two connections on $V$. In \cite{CS}, S.S. Chern and J. Simons introduced secondary characteristic forms --- 
the Chern-Simons forms $\mathrm{cs}(\nabla^{1},\nabla^{0})$. They are defined modulo exact forms, satisfy the equation 
\begin{equation} \label{cs-1}
d\,\mathrm{cs}(\nabla^{1},\nabla^{0})=\ch(V,\nabla^{1})-\ch(V,\nabla^{0}),
\end{equation} 
and enjoy a functoriality property under the pullbacks with smooth maps. 
When the bundle $V$ is flat, putting $\nabla^{1}=d+A$ and $\nabla^{0}=d$ and using linear homotopy $A(t)=tA$ in the Chern-Weil homotopy formula, one obtains an explicit formula for the Chern-Simons form $\mathrm{cs}(A)$ in terms of $A$.

Let $(E,h)$ be a holomorphic Hermitian vector bundle --- a holomorphic vector bundle of rank $r$ over a complex manifold $X$, $\dim_{\C}X=n$, with a Hermitian metric $h$. Chern-Weil theory associates to any polynomial $\Phi$ on GL$(r,\mathbb{C})$, invariant under conjugation, a differential form $\Phi(\bm\Theta)$ on $X$. Special case of this construction is the  Chern character form $\ch(E,h)$, defined by 
$$\ch(E,h)=\Tr\left\{\exp\left(\frac{\sqrt{-1}}{2\pi}\,\bm{\Theta}\right)\right\}=\sum_{k=0}^{n}\ch_{k}(E,h).$$
Here $\bm{\Theta}$ is the curvature of the canonical connection $d+\bm{\theta}$ in $E$ associated with the Hermitian metric $h$. In the local holomorphic frame, $\bm{\theta}=h^{-1}\del h$ and $\bm{\Theta}=\delb\bm{\theta}$ (see, e.g., \cite{GH}).

Let $h_{1}$ and $h_{2}$ be two Hermitian metrics on a holomorphic vector bundle $E$ over a complex manifold $X$. In the classic paper \cite{BC}, Bott and Chern showed the existence of certain secondary characteristic forms, the Bott-Chern secondary forms $\mathrm{bc}(E,h_{1},h_{2})$. They are defined modulo $\del$ and $\delb$-exact forms, satisfy the equation 
$$\frac{\sqrt{-1}}{2\pi}\delb\del\,\mathrm{bc}(E,h_{1},h_{2})= \ch(E,h_{1})-\ch(E,h_{2})$$
and enjoy the functoriality property with respect to the  pullbacks by holomorphic maps. Here the Chern character forms are computed for canonical connections in $(E,h_{1})$ and $(E,h_{2})$. The Bott-Chern forms have been used in geometric stability \cite{D,Tian}, in higher dimensional Arakelov geometry \cite{GS1,Tam} and in physics \cite{LMNS} (see also \cite{PT} for their application to differential $K$-theory).

However, it is difficult to obtain explicit formulas for the Bott-Chern forms.  It is already mentioned in the remark in  \cite[Sect. 3]{BC} that even for a linear homotopy $h_{t}$ of Hermitian metrics, the homotopy
formula in Proposition 3.15 in \cite{BC} contains the inverse metrics through $\bm\Theta_{t}=\delb (h_{t}^{-1}\del h_{t})$ and does not allow to integrate over $t$ in a closed form.
As the result, it is difficult\footnote{As was observed in \cite{D}, ``One interesting feature is that we have an example of a variational problem with no
simple explicit formula for the Lagrangian''.}  to get explicit formulas for the Bott-Chern forms in terms of Hermitian metrics $h_{1}$ and $h_{2}$ only. 
This problem manifests itself even for the case when $E$ is a trivial bundle with metrics $h_{1}=h$ and $h_{2}=I$, the identity matrix. 

Here we show how using global coordinates on the space of Hermitian positive-definite matrices associated with the Cholesky decomposition, one can obtain explicit formulas for the Bott-Chern forms on trivial bundles.
Namely, in Proposition \ref{D-1} we present explicit decomposition of the Chern-Simons form $\mathrm{cs}_{k}$ associated to the Chern character form  $\ch_{k}=\ch_{k}(E,h)$ into $(p,q)$-degrees. It is done in Section \ref{2} by solving the descent equations from the double complex of $(p,q)$-forms on $X$, applied to $\ch_{k}$.
In Section \ref{3} we introduce the `ascent' equations that start from the $(2k-1,0)$ component of  $\mathrm{cs}_{k}$, and use Cholesky decomposition of the Hermitian metric $h$ to represent the Chern-Simons form, modulo $d$-exact forms, as a $\del$-exact form. This yields an explicit formula for the Bott-Chern form $\mathrm{bc}_{k}$ of type $(k-1,k-1)$ such that $\displaystyle{\ch_{k}=\frac{\sqrt{-1}}{2\pi}\delb\del\mathrm{bc}_{k}}$. It is obtained by repeatedly finding corresponding $\del$-antiderivatives and seems to be very non-local. For the case $k=2,3$ in Propositions \ref{2} and \ref{3} we present explicit local
formulas for these forms in Cholesky coordinates. We believe that such explicit formulas exist for all $k$. In Remark \ref{positive} we prove that the form $\mathrm{bc}_{2}$ is positive, and in Remark \ref{Donaldson} we directly show that for bundles with upper-triangular transition functions the Euler-Lagrange functional $M_{C}(-,K)$ introduced in \cite{D} is bounded below.
\subsection*{Acknowledgments} This work was done under partial support of the NSF grant DMS-1005769. I am grateful to Vamsi Pingali for stimulating discussions and suggestions, and thank Kiyoshi Igusa for his remarks.
\section{Double descent} \label{2}
\subsection{Set-up} Let $h$ be Hermitian metric in rank $r$ trivial complex vector bundle over a complex manifold $X$ (i.e., in general we consider a local holomorphic frame over some open neighborhood). Put (see, e.g., \cite{GH})
$$\bm\theta=h^{-1}\del h\quad\text{and}\quad\bm{\Theta}=\delb\bm\theta. $$
We have the following useful formulas
\begin{equation} \label{tT-1}
\del\bm{\theta}=-\bm{\theta}^{2},\quad \delb\bm{\theta}=\bm{\Theta} \quad\text{and}\quad \del\bm{\Theta}=[\bm{\Theta},\bm{\theta}],\quad \delb\bm{\Theta}=0,
\end{equation}
where for matrix-valued differential forms we write $AB$ instead of $A\wedge B$, etc. In particular, we have
\begin{equation} \label{tT-2}
\del\bm{\Theta}^{k}=[\bm{\Theta}^{k},\bm{\theta}]\quad\text{and}\quad \del(\bm{\theta}\bm{\Theta}^{k})=-\bm{\theta}\bm{\Theta}^{k}\bm{\theta}.
\end{equation}

We have (using a `system of units' such that $\displaystyle{\frac{\sqrt{-1}}{2\pi}=1}$) 
$$\ch_{k}(h)=\frac{1}{k!}\omega_{k,k},$$
where
$$\omega_{k,k}=\Tr\bm{\Theta}^{k}$$
is $\del$ and $\delb$-closed real form of type $(k,k)$; here and in what follows $\omega_{p,q}$ denotes a $(p,q)$-form. It follows from Poincar\'{e} lemma that \emph{locally} (i.e., on some polydisk coordinate chart of $X$) there are forms $\omega_{k+l,k-l-1}$ such that 
\begin{align*}
\omega_{k,k}& =\delb\omega_{k,k-1},\\
\del\omega_{k,k-1}&=\delb\omega_{k+1,k-2},\\
&\,\vdots\\
\del\omega_{2k-2,1}&=\delb\omega_{2k-1,0},\\
\del\omega_{2k-1,0}&=0.
\end{align*}

These descent equations\footnote{Compare with the double descent in \cite{FS} and with the holomorphic descent in \cite{LMNS}.} can be written succinctly as

\begin{equation} \label{descent}
\omega_{k,k}=(\delb-t\del)\!\left(\omega_{k,k-1}+t\omega_{k+1,k-2}+\dots+t^{k-2}\omega_{2k-2,1}+t^{k-1}\omega_{2k-1,0}\right).
\end{equation}
\begin{remark} Putting $t=-1$ we get
$$\omega_{k,k}=d\!\left(\omega_{k,k-1}-\omega_{k+1,k-2}+\dots+(-1)^{k-2}\omega_{2k-2,1}+(-1)^{k-1}\omega_{2k-1,0}\right),$$
where $d=\del+\delb$. This gives an explicit decomposition of the Chern-Simons secondary form $\mathrm{cs}_{k}$ into $(p,q)$-degrees, $p+q=2k-1$:
\begin{equation} \label{cs}
\mathrm{cs}_{k}=\frac{1}{k!}\!\left(\!\omega_{k,k-1}-\omega_{k+1,k-2}+\dots+(-1)^{k-2}\omega_{2k-2,1}+(-1)^{k-1}\omega_{2k-1,0}\!\right)\!.
\end{equation}
\end{remark}
It is easy to compute all these forms using \eqref{tT-1}--\eqref{tT-2} and equations
\begin{equation} \label{tT-3}
\del{\bm\ta}^{2}=0,\quad\delb\bm{\ta}^{2}=[\bm{\Theta},\bm{\theta}].
\end{equation}
First, we observe
$$\omega_{k,k-1}=\Tr(\tam\Tam^{k-1})$$
and state the following result.
\begin{lemma} \label{L-1} We have
$$\del\omega_{k,k-1}=\Tr(\tam^{2}\Tam^{k-1})=\delb\omega_{k+1,k-2},$$
where
$$\omega_{k+1,k-2}=\frac{1}{k+1}\Tr\!\left\{\tam\!\left(\tam^{2}\Tam^{k-2}+\Tam\tam^{2}\Tam^{k-3}+\cdots+\Tam^{k-3}\tam^{2}\Tam +\Tam^{k-2}\tam^{2}\right)\right\}.$$
\end{lemma}
\begin{proof} Using \eqref{tT-1}--\eqref{tT-2}, we get
$$\del\Tr(\tam\Tam^{k-1})=\Tr(-\tam^{2}\Tam^{k-1}-\tam(\Tam^{k-1}\tam-\tam\Tam^{k-1}))=-\Tr(\tam\Tam^{k-1}\tam)=\Tr\tam^{2}\Tam^{k-1}.$$
Next, using \eqref{tT-1} and \eqref{tT-3}, we get
\begin{align*}
\delb\sum_{i=0}^{k-2}(\tam\Tam^{i}\tam^{2}\Tam^{k-2-i})& =\sum_{i=0}^{k-2}(\Tam^{i+1}\tam^{2}\Tam^{k-2-i})-\sum_{i=0}^{k-2}(\tam\Tam^{i}(\Tam\tam-\tam\Tam)\Tam^{k-2-i})\\
&=\sum_{i=0}^{k-2}(\Tam^{i+1}\tam^{2}\Tam^{k-2-i}) +\tam^{2}\Tam^{k-1} -\tam\Tam^{k-1}\tam,
\end{align*}
since the second sum telescopes. Using the cyclic property of the trace, we get the formula for $\omega_{k+1,k-2}$.
\end{proof}

Observe that $\omega_{k,k-1}$ is a constant term $a_{0}$ in the polynomial
\begin{equation}\label{F}
F_{k}(t)=\Tr\!\left\{\tam(\Tam +t\tam^{2})^{k-1}\right\}=a_{0}+a_{1}t+\cdots + a_{k-1}t^{k-1},
\end{equation}
while by Lemma \ref{L-1},
$$\displaystyle{\omega_{k+1,k-2}=\frac{1}{k+1}a_{1}}.$$ 
This suggests to consider all coefficients $a_{l}$ of $F_{k}(t)$ --- differential forms of types $(k+l,k-l-1)$, $l=0,1,\dots,k-1$.

\begin{lemma}\label{L-2} Put $G_{k}(t)=\Tr(\Tam+t\tam^{2})^{k}$. We have
$$\delb F_{k}(t) -t\del F_{k}(t)=G_{k}(t).$$
\end{lemma}
\begin{proof} It follows from equations \eqref{tT-1} and \eqref{tT-3} that
$$\del(\Tam+t\tam^{2})=[(\Tam+t\tam^{2}),\tam]\quad\text{and}\quad \delb(\Tam+t\tam^{2})=t[(\Tam+t\tam^{2}),\tam],$$
which implies
$$\del(\Tam+t\tam^{2})^{k}=[(\Tam+t\tam^{2})^{k},\tam]\quad\text{and}\quad \delb(\Tam+t\tam^{2})^{k}=t[(\Tam+t\tam^{2})^{k},\tam].$$
Therefore,
\begin{align*}
\del F_{k}(t)& =\Tr\!\left\{-\tam^{2}(\Tam +t\tam^{2})^{k-1}-\tam\!\left((\Tam +t\tam^{2})^{k-1}\tam-\tam(\Tam +t\tam^{2})^{k-1}\right)\right\}\\
& =\Tr\!\left\{\tam^{2}(\Tam +t\tam^{2})^{k-1}\right\}
\end{align*}
and
\begin{align*}
\delb F_{k}(t)& =\Tr\!\left\{\Tam(\Tam +t\tam^{2})^{k-1}-t\tam\!\left((\Tam +t\tam^{2})^{k-1}\tam-\tam(\Tam +t\tam^{2})^{k-1}\right)\right\}\\
& =t\Tr\!\left\{\tam^{2}(\Tam +t\tam^{2})^{k-1}\right\} + \Tr(\Tam+t\tam^{2})^{k},
\end{align*}
so that $(\delb-t\del)F_{k}(t)=G_{k}(t)$.
\end{proof}

From here it is easy to find all descent forms $\omega_{k+l,k-l-1}$.
\begin{proposition} \label {D-1} We have 
$$\mathrm{cs}_{k}=\frac{1}{k!}\!\left(\!\omega_{k,k-1}-\omega_{k+1,k-2}+\dots+(-1)^{k-2}\omega_{2k-2,1}+(-1)^{k-1}\omega_{2k-1,0}\!\right)\!,$$
where
$$\omega_{k+l,k-l-1}=\frac{k!l!}{(k+l)!}a_{l},\quad l=0,1,\dots,k-1.$$
In particular,
$$\omega_{2k-1,0}=\frac{k!(k-1)!}{(2k-1)!}\Tr\bm\theta^{2k-1}.$$
\end{proposition}
\begin{proof} We observe that
$$\frac{d G_{k}}{dt}(t)=k\Tr\!\left\{\tam^{2}(\Tam +t\tam^{2})^{k-1}\right\}=k\del F_{k}(t),$$
so that
$$G_{k}(t)=b_{0}+k\del a_{0}\frac{t}{1}+k\del a_{1} \frac{t^{2}}{2}+\dots+k\del a_{k-1}\frac{t^{k}}{k},$$
where $b_{0}=\Tr\Tam^{k}=\omega_{k,k}$. Now it follows from Lemma \ref{L-2} that
$$\delb a_{l}=\left(\frac{k+l}{l}\right)\del a_{l-1},\quad l=1,\dots,k-1,$$
and since $a_{0}=\omega_{k,k-1}$, we easily obtain 
$$a_{l}=\frac{(k+l)\dots(k+1)}{l!}\omega_{k+l,k-l-1}.\qedhere$$
\end{proof}

Thus for $k=1$ we have 
$$\omega_{1,0}=\Tr\bm{\theta}=\del\log\det h\quad\text{and}\quad\omega_{0,0}=\log\det h, $$
whereas for $k=2$
$$\omega_{2,1}=\Tr(\bm{\theta}\bm{\Theta})\quad\text{and}\quad\omega_{3,0}=\frac{1}{3}\Tr\bm{\theta}^{3}.$$
For $k=3$ we have
$$\omega_{3,2}=\Tr(\bm{\theta}\bm{\Theta}^{2}), \quad \omega_{4,1}=\frac{1}{2}\Tr(\bm{\theta}^{3}\bm{\Theta})\quad\text{and}\quad
\omega_{5,0}=\frac{1}{10}\Tr\bm{\theta}^{5},$$
and for $k=4$ from Proposition \ref{D-1} we obtain
$$\omega_{4,3}=\Tr(\bm\ta\bm{\Theta}^{3}),\quad\omega_{5,2}=\frac{1}{5}\Tr\!\left(\bm\ta^{3}\bm{\Theta}^{2}+\bm\ta\bm{\Theta}\bm\ta^{2}\bm{\Theta}+\bm\ta\bm{\Theta}^{2}\bm\ta^{2}\right),\quad
\omega_{6,1}=\frac{1}{5}\Tr(\bm\ta^{5}\bm{\Theta})$$
and
$$\omega_{7,0}=\frac{1}{35}\Tr\bm\ta^{7}.$$

\begin{remark} The forms $\Tr\bm\ta^{2k-1}$, $k\geq 1$, where $\bm\theta=g^{-1}dg$ is a Maurer-Cartan form,
generate the cohomology ring $H^{\bullet}(\mathrm{GL}(\infty,\CC),\mathbb{Q})$ for the stabilized complex general linear group
$\mathrm{GL}(\infty,\CC)$. 
\end{remark}

\section{Double ascent} \label{3}
\subsection{Set-up}
From descent equations it follows that there is a form $\omega_{2k-2,0}$ such that
$$\omega_{2k-1,0}=\del\omega_{2k-2,0}.$$
Now going up from the bottom to the top (this explains the terminology), we get 
$$\del(\omega_{2k-2,1}+\delb\omega_{2k-2,0})=0,$$
so that there is a form $\omega_{2k-3,1}$ such that
$$\omega_{2k-2,1}+\delb\omega_{2k-2,0}=\del\omega_{2k-3,1}.$$
Therefore 
$$\del(\omega_{2k-3,2}+\delb\omega_{2k-3,1})=0$$
and there is a form $\omega_{2k-4,2}$ such that
$$\omega_{2k-3,2}+\delb\omega_{2k-3,1}=\del\omega_{2k-4,2}.$$
Repeating this procedure, we finally get a form $\omega_{k-1,k-1}$ such that
$$\omega_{k,k-1}+\delb\omega_{k,k-2}=\del\omega_{k-1,k-1}.$$

The ascent equations can be written succinctly as
\begin{equation*}
k!\mathrm{cs}_{k}=\del\omega_{k-1,k-1}-d\left(\omega_{k,k-2}-\omega_{k+1,k-3}+\cdots + (-1)^{k}\omega_{2k-2,0}\right).
\end{equation*}
Defining $\mathrm{CS}_{k}$ as $\mathrm{cs}_{k}\!\!\!\mod \mathrm{Im}\,d$ (see \cite{SS}), we get
\begin{equation}
\mathrm{CS}_{k}=\frac{1}{k!}\del\omega_{k-1,k-1}.
\end{equation}
Therefore,
$$\ch_{k}=\frac{1}{k!}\delb\del\omega_{k-1,k-1},$$
so that $\omega_{k-1,k-1}$ is a $k!$ times the Bott-Chern secondary form $\mathrm{bc}_{k}$ (see \cite{BC}).

\begin{remark}
As a corollary, we have the following version of local ``$\del\delb$ lemma'': for each form $\omega$ of type $(k,k)$ on a complex manifold $X$ satisfying $d\omega=0$ on every polydisk neighborhood $U\subset X$ there is a form $\theta_{U}$ on $U$ such that $\left.\omega\right|_{U}=\delb\del\theta_{U}$.
\end{remark}

Solving `explicitly' ascent equations would give explicit local expression of the Chern character form $\ch_{k}$ in terms of the corresponding Bott-Chern form $\mathrm{bc}_{k}$. It is known that it is not possible to get local formulas in terms of the matrix $h$ alone. This is because each step in the ascent procedure uses Poincar\'{e} lemma which, in general, contains an integration through the homotopy formula. However, one can solve these equations explicitly by using the Cholesky decomposition! 

Namely, put 
$$h=cb=b^{\ast}ab,$$ 
where matrix $b$ is upper-triangular with $1$'s on the diagonal, and $a$ is diagonal with positive entries; $a_{i}$ and $b_{ij}$, $i=1,\dots,r, j>i$, are global coordinates on the homogeneous space of hermitian positive-definite $r\times r$ matrices. We get
\begin{equation} \label{bold-t}
\bm{\theta}=h^{-1}\del h=b^{-1}\ta b=b^{-1}(\theta_{1}+\theta_{2})b,
\end{equation}
where 
$$\theta_{1}=\del b\,b^{-1}\quad \text{and}\quad \theta_{2}=c^{-1}\del c.$$
Therefore
\begin{equation}\label{bold-tT-1}
\bm{\Theta}=\delb\bm\ta=b^{-1}\left(\delb\theta-\tab_{1}\ta-\ta\tab_{1}\right)b,
\end{equation}
where
$$\tab_{1}=\delb b\,b^{-1}\quad\text{and}\quad\tab_{2}=c^{-1}\delb c.$$
These matrix-valued $1$-forms satisfy 
\begin{gather}
\del\ta_{1}=\ta_{1}^{2}, \quad\del\ta_{2}=-\ta_{2}^{2},\quad\delb\tab_{1}=\tab_{1}^{2},\quad\delb\tab_{2}=-\tab_{2}^{2},\label{t1} \\
\delb\ta_{1}=-\del\tab_{1} +\ta_{1}\tab_{1}+\tab_{1}\ta_{1},\quad \delb\ta_{2}=-\del\tab_{2} -\ta_{2}\tab_{2}-\tab_{2}\ta_{2} \label{t2}.
\end{gather}
Moreover, since $\ta_{1}$ is nilpotent and
$$\ta_{2}=a^{-1}\ta_{1}^{\ast}a+a^{-1}\del a,$$
we have important property
\begin{equation} \label{trace}
\Tr\left(\ta_{1}^{l_{1}}\tab_{1}^{\bar{l}_{1}}\right)=\Tr\left(\ta_{2}^{l_{2}}\tab_{2}^{\bar{l}_{2}}\right)=0
\end{equation}
for all $l_{1}+\bar{l}_{1}>0$ and $ l_{2}+\bar{l}_{2}>1$. We will also be using 
\begin{equation} \label{bold-tb}
\bm\tab=h^{-1}\delb h=b^{-1}\tab b=b^{-1}(\tab_{1}+\tab_{2})b,
\end{equation}
so that
\begin{equation} \label{bold-T-tb}
\delb\bm\tab=-\bm\tab^{2}\quad\text{and}\quad\bm{\Theta}=-\del\bm\tab-\bm\ta\bm\tab-\bm\tab\bm\ta.
\end{equation}

It turns out that in terms of $\ta_{1},\tab_{1},\ta_{2},\tab_{2}$ and its $\del$ and $\delb$ differentials one can explicitly compute differential forms $\omega_{2k-2-l,l}$ for $l=0,1,\dots, k-1$. 

\begin{remark} \label{Triangular} The Cholesky decomposition is useful since by the holomorphic splitting principle (see, e.g., \cite[Corollary 9.26]{CV}),
for every holomorphic vector bundle $E\rightarrow X$ there exists a variety $Y$ and a flat morphism $p: Y\rightarrow X$ such that the bundle $p^{\ast}(E)$ over $Y$ admits upper-triangular transition functions.
\end{remark}
\subsection{The case $k=2$} Start with the form $\omega_{3,0}=\frac{1}{3}\Tr\bm\ta^{3}$. In terms of the Cholesky decomposition we have
$$\omega_{3,0}=\frac{1}{3}\Tr(\ta_{1}^{3}+ 3\ta^{2}_{1}\ta_{2}+3\ta_{1}\ta_{2}^{2}+\ta_{2}^{3})=\Tr(\ta^{2}_{1}\ta_{2}+\ta_{1}\ta_{2}^{2})=\del\Tr(\theta_{1}\theta_{2}),$$
so that
$$\omega_{2,0}=\Tr(\theta_{1}\theta_{2}).$$
Using \eqref{bold-T-tb} we get
\begin{align*}
\omega_{2,1}& =\Tr(\bm\ta\bm{\Theta})=-\Tr(\bm\ta(\del\bm\tab +\bm\ta(\bm\ta\bm\tab+\bm\tab\bm\ta))=\del\Tr(\bm\ta\bm\tab)+\Tr(\bm\ta^{2}\bm\tab-\bm\ta(\bm\ta\bm\tab+\bm\tab\bm\ta))\\
&=\del\Tr(\ta\tab)-\Tr(\ta^{2}\tab),
\end{align*}
and using \eqref{t2} we obtain
\begin{align*}
\omega_{2,1}+\delb\omega_{2,0} & =\del\Tr(\ta\tab) + \Tr(-\ta^{2}\tab +(\delb\ta_{1}\ta_{2}-\ta_{1}\delb\ta_{2})) \\
& = \del\Tr(\ta\tab)+ \Tr(-\ta^{2}\tab -\del\tab_{1}\ta_{2}+\ta_{1}\del\tab_{2} +(\ta_{1}\ta_{2}+\ta_{2}\ta_{1})\tab)\\
&= \del\Tr(\ta\tab-(\tab_{1}\ta_{2}-\tab_{2}\ta_{1})) \\
&\quad + \Tr(-\ta^{2}\tab +\ta_{1}^{2}\tab_{2}+\tab_{2}\tab_{1}^{2} +(\ta_{1}\ta_{2}+\ta_{2}\ta_{1})\tab)\\
&= \del\Tr(\ta\tab-(\tab_{1}\ta_{2}-\tab_{2}\ta_{1})).
\end{align*}
Thus
\begin{equation} \label{1-1}
\omega_{1,1}=\Tr\!\left(\ta\tab-(\tab_{1}\ta_{2}-\tab_{2}\ta_{1})\right)=\Tr\!\left(2\ta_{2}\tab_{1} +\ta_{2}\tab_{2}\right),
\end{equation}
and we obtain the following result.
\begin{proposition} \label{2} The second Bott-Chern form $\mathrm{bc}_{2}$ of a trivial Hermitian vector bundle $(\CC^{r},h)$ over a complex manifold $X$ in Cholesky coordinates $h=b^{*}ab$ is given by the formula
$$\mathrm{bc}_{2}=\frac{1}{2}\Tr\!\left(2\ta_{2}\tab_{1} +\ta_{2}\tab_{2}\right).$$
Here  $\tab_{1}=\delb b b^{-1}, \theta_{2}=c^{-1}\del c$ and $\tab_{2}=c^{-1}\delb c$.
\end{proposition}
\begin{remark} \label{positive} Using that $c=b^{\ast}a$, we obtain from \eqref{1-1} that
$$\omega_{1,1}=\Tr\!\left(a^{-1}\del a\wedge a^{-1}\delb a + 2\varphi\wedge\varphi^{\ast}\right),$$
where $\varphi=a^{-1/2}(b^{\ast})^{-1}\del b^{\ast}a^{1/2}$, so that $\sqrt{-1}\,\omega_{1,1}\geq 0$.
\end{remark}
\begin{remark} When 
$$a=\begin{pmatrix}1&0\\0 & e^{\sigma}\end{pmatrix}\quad\text{and}\quad b=\begin{pmatrix}1&\bar{f}\\0 & 1\end{pmatrix},$$
we get
$$\omega_{0,0}=\sigma\quad\text{and}\quad\omega_{1,1}=\Tr\!\left(\del\sigma\wedge\delb\sigma + 2e^{-\sigma}\del f\wedge\delb\bar{f}\right),$$
so that 
$$\frac{1}{2}\delb\del\omega_{1,1} -\frac{1}{2}(\delb\del\omega_{0,0})^{2}=\delb\del\left(e^{-\sigma}\del f\wedge\delb\bar{f}\right),$$
in agreement with Remark 3.4 in \cite{PT}.
\end{remark}

Following Remark \ref{Triangular}, consider rank $r$ Hermitian vector bundle $(E,h)$ with transition functions taking values in the Borel subgroup $B(r)$ of upper-triangular matrices in $\mathrm{GL}(r,\CC)$. In terms of a local trivialization of $E$ --- an open cover $\{U_{\alpha}\}$ of $X$ and holomorphic transition functions $g_{\alpha\beta}: U_{\alpha}\cap U_{\beta}\rightarrow B(r)$, a Hermitian metric $h$ on $E$ is given by a collection $\{h_{\alpha}\}$ of positive-definite Hermitian matrices on $U_{\alpha}$, satisfying 
$$h_{\beta}=g_{\alpha\beta}^{\ast}h_{\alpha}g_{\alpha\beta}\quad\text{on}\quad U_{\alpha}\cap U_{\beta}.$$
Denote by $\mathrm{bc}_{2\alpha}$ the second Bott-Chern form on $U_{\alpha}$ and write $g_{\alpha\beta}=a_{\alpha\beta}b_{\alpha\beta}$, where
 $a_{\alpha\beta}$ are diagonal and $b_{\alpha\beta}$ are unipotent.  From Proposition \ref{2} we obtain
\begin{equation} \label{charts}
\mathrm{bc}_{2\beta}=\mathrm{bc}_{2\alpha}+c_{\alpha\beta}\quad\text{on}\quad U_{\alpha}\cap U_{\beta},
\end{equation}
where
$$c_{\alpha\beta}=\Tr\left\{a_{\alpha\beta}^{-1}\del a_{\alpha\beta}\wedge (\overline{a_{\alpha\beta}^{-1}\del a_{\alpha\beta}}) +a_{\alpha\beta}^{-1}\del a_{\alpha\beta}\wedge a_{\alpha}^{-1}\delb a_{\alpha}+a_{\alpha}^{-1}\del a_{\alpha}\wedge (\overline{a_{\alpha\beta}^{-1}\del a_{\alpha\beta}})  \right\},$$
and depends only on $a_{\alpha\beta}$.
Since $a_{\alpha\beta}$ are holomorphic, we have $\delb\del c_{\alpha\beta}=0$. In particular, if transition functions are unipotent, it follows from \eqref{charts} that local expressions $\{\mathrm{bc}_{2\alpha}\}$ determine a well-defined $(1,1)$-form on $X$.

\begin{remark} \label{Donaldson}
Given two Hermitian metrics $h_1$ and $h_2$ on a holomorphic vector bundle $E$, we define a local Bott-Chern form
$\mathrm{bc}_{2}(h_1,h_2)$ by
$$\mathrm{bc}_{2}(h_1,h_2)=\mathrm{bc}_{2}(h_{1})-   \mathrm{bc}_{2}(h_{2}),$$
where $\mathrm{bc}_{2}(h_{1,2})$ are given in Proposition \ref{2} with $h=h_{1,2}$. It follows from \eqref{charts} that for the bundle $E$
with upper-triangular transition functions $\mathrm{bc}_{2}(h_1,h_2)$ is a well-defined $(1,1)$-form on $X$. In particular, for such bundles Proposition \ref{2} provides an explicit formulas for the functionals $M_{\omega}( -,K)$ and $M_{C}(-,K)$ in Donaldson's paper \cite{D}, and from Remark \ref{positive} one gets that $M_{C}(-,K)$ is bounded below \cite[Corollary 9]{D}.
\end{remark}

\begin{remark} Upper triangular matrices were used for the study the higher Reidemeister torsion in \cite{Igusa}. Though the set-up in this paper and in \cite{Igusa} is different, it would be interesting to compare corresponding calculations.
\end{remark}

\subsection{The case $k=3$} Using \eqref{trace} we get 
\begin{align*}
\omega_{5,0} & =\frac{1}{10}\Tr\bm\theta^{5}=\frac{1}{10}\Tr\ta^{5} \\
& =
\frac{1}{2}\Tr\!\left(\ta_{1}^{4}\ta_{2}+\ta_{1}^{3}\ta_{2}^{2}+\ta_{1}^{2}\ta_{2}\ta_{1}\ta_{2}+\ta_{1}\ta_{2}\ta_{1}\ta_{2}^{2}+\ta_{1}^{2}\ta_{2}^{3} +\ta_{1}\ta_{2}^{4}\right) \\
&=\frac{1}{2}\del\Tr\!\left(\ta_{1}^{3}\ta_{2}+\ta_{1}\ta_{2}^{3}+\frac{1}{2}(\ta_{1}\ta_{2})^{2} \right),
\end{align*}
so that
$$\omega_{4,0}=\frac{1}{2}\Tr\!\left(\ta_{1}^{3}\ta_{2}+\ta_{1}\ta_{2}^{3}+\frac{1}{2}(\ta_{1}\ta_{2})^{2} \right).$$

We will compute $\omega_{4,1}+\delb\omega_{4,0}$ and will find $\omega_{3,1}$ such that 

$$\omega_{4,1}+\delb\omega_{4,0}=\del\omega_{3,1}.$$ 
First using \eqref{bold-T-tb} we get
\begin{align*}
\omega_{4,1} &=\frac{1}{2}\Tr(\bm\theta^{3}\bm{\Theta})=-\frac{1}{2}\Tr(\bm\theta^{3}(\del\bm\tab+\bm\tab\bm\ta+\bm\ta\bm\tab))\\
& = \frac{1}{2}\del\Tr(\bm\ta^{3}\bm\tab)+\frac{1}{2}\Tr(\bm\ta^{4}\bm\tab -\bm\ta^{3}(\bm\tab\bm\ta+\bm\ta\bm\tab))\\
& =\frac{1}{2}\del\Tr(\ta^{3}\tab)-\frac{1}{2}\Tr(\ta^{4}\tab).
\end{align*}
Next, using \eqref{t2} we obtain 
\begin{gather*}
\delb\omega_{4,0} =\frac{1}{2}\Tr\!\left(I_{1}\delb\ta_{1}+I_{2}\delb\ta_{2}\right)\\
 = \frac{1}{2}\Tr\!\left(I_{1}(-\del\tab_{1}+\ta_{1}\tab_{1}+\tab_{1}\ta_{1})+I_{2}(-\del\tab_{2} -\ta_{2}\tab_{2}-\tab_{2}\ta_{2})\right)\\
=\frac{1}{2}\del\Tr\!\left(I_{1}\tab_{1}+I_{2}\tab_{2}\right)+\frac{1}{2}\Tr\!\left((-\del I_{1} +I_{1}\ta_{1}+\ta_{1}I_{1})\tab_{1}-(\del I_{2}+I_{2}\ta_{2}+\ta_{2}I_{2})\tab_{2}\right),
\end{gather*}
where
\begin{align*}
I_{1} & = \ta_{1}^{3}+ \ta_{1}^{2}\ta_{2}+\ta_{2}\ta_{1}^{2}-\ta_{1}\ta_{2}\ta_{1}+\ta_{2}\ta_{1}\ta_{2}+\ta_{2}^{3}=\ta^{3}-\ta\ta_{2}\ta_{1}-\ta_{1}\ta_{2}\ta\\
\intertext{and}
I_{2}& = -(\ta_{1}^{3}+\ta_{1}\ta_{2}\ta_{1}+\ta_{1}\ta_{2}^{2}-\ta_{2}\ta_{1}\ta_{2}+\ta_{2}^{2}\ta_{1} +\ta_{2}^{3})
=-\ta^{3} +\ta\ta_{1}\ta_{2}+\ta_{2}\ta_{1}\ta.
\end{align*}
Using identities
\begin{equation*}
\del I_{1}-I_{1}\ta_{1}-\ta_{1}I_{1}  = -\ta^{4}\quad\text{and}\quad
\del I_{2} +I_{2}\ta_{2}+\ta_{2}I_{2}  =-\ta^{4},
\end{equation*}
we get
$$\omega_{4,1}+\delb\omega_{4,0}=\frac{1}{2}\del\Tr\!\left(\ta^{3}\tab+I_{1}\tab_{1}+I_{2}\tab_{2}\right),$$
so that
$$\omega_{3,1}=\frac{1}{2}\Tr\!\left(\ta^{3}\tab+I_{1}\tab_{1}+I_{2}\tab_{2}\right).$$
Equivalently,
$$\omega_{3,1}=\frac{1}{2}\Tr\!\left(2\ta^{3}\tab_{1}-(\ta_{1}\ta_{2}^{2}+2\ta_{1}\ta_{2}\ta_{1}+\ta_{2}^{2}\ta_{1})\tab_{1}+(\ta_{1}^{2}\ta_{2}+2\ta_{2}\ta_{1}\ta_{2}+\ta_{2}\ta_{1}^{2})\tab_{2} \right). $$

Finally, we will compute $\omega_{3,2}+\delb\omega_{3,1}$ and find $\omega_{2,2}$ such that 

$$\omega_{3,2}+\delb\omega_{3,1}=\del\omega_{2,2}.$$ 

First, using \eqref{bold-T-tb} we obtain
\begin{align*}
\del\Tr(\bm\ta\bm\Theta\bm\tab) & =\Tr\!\left(-\bm\ta^{2}\bm\Theta\bm\tab -\bm\ta(\bm\Theta\bm\ta-\bm\ta\bm\Theta)\bm\tab+\bm\ta\bm\Theta(\bm{\Theta}+\bm\ta\bm\tab+\bm\tab\bm\ta)\right) \\
& = \Tr\!\left(\bm\ta\bm\Theta^{2} +\bm\ta^{2}\bm\Theta\bm\tab \right), \\
\intertext{and}
\del\Tr(\bm\tab\bm\Theta\bm\ta) & =\Tr\!\left(-(\bm{\Theta}+\bm\ta\bm\tab+\bm\tab\bm\ta)\bm\Theta\bm\ta -\bm\tab(\bm\Theta\bm\ta-\bm\ta\bm\Theta)\bm\ta+\bm\tab\bm\Theta\bm\ta^{2}\right)\\
& =-\Tr\!\left(\bm\ta\bm\Theta^{2} +\bm\tab\bm\Theta\bm\ta^{2} \right), 
\end{align*}
so that
$$\omega_{3,2}=\Tr(\bm\ta\bm\Theta^{2})=\del\left\{\frac{1}{2}\Tr(\bm\ta\bm\Theta\bm\tab-\bm\tab\bm\Theta\bm\ta)\right\}-\frac{1}{2}\Tr(\bm\ta^{2}\bm\Theta\bm\tab+ \bm\tab\bm\Theta\bm\ta^{2}).$$
Next, we write
$$\omega_{3,1}=\frac{1}{2}\Tr\!\left(\bm\ta^{3}\bm\tab +\bm I_{1}\bm\tab_{1} +\bm I_{2}\bm\tab_{2}\right)=\omega_{3,1}^{(1)}+\omega_{3,1}^{(2)},$$
where $\bm\tab_{1}=b^{-1}\tab_{1}b$ and $\bm\tab_{2}=b^{-1}\tab_{2}b$  and 
$$\bm I_{1} =\bm\ta^{3}-\bm\ta\bm\ta_{2}\bm\ta_{1}-\bm\ta_{1}\bm\ta_{2}\bm\ta, \quad \bm I_{2} =
-\bm\ta^{3} +\bm\ta\bm\ta_{1}\bm\ta_{2}+\bm\ta_{2}\bm\ta_{1}\bm\ta,$$
where $\bm\ta_{1}=b^{-1}\ta_{1}b$ and $\bm\ta_{2}=b^{-1}\ta_{2}b$.
We have
\begin{align*}
\delb\omega_{3,1}^{(1)} &=\frac{1}{2}\delb\Tr(\bm\theta^{3}\bm\tab) \\
&=\frac{1}{2}\Tr(\bm\Theta\bm\ta^{2}\bm\tab-\bm\ta\bm\Theta\bm\ta\bm\tab +\bm\ta^{2}\bm\Theta\bm\tab+\bm\theta^{3}\bm\tab^{2}) \\
&=\frac{1}{2}\Tr(\bm\tab\bm\Theta\bm\ta^{2}+\bm\ta^{2}\bm\Theta\bm\tab-\bm\ta\bm\Theta\bm\ta\bm\tab +\bm\theta^{3}\bm\tab^{2}),
\end{align*}
so that
$$\omega_{3,2}+\delb\omega_{3,1}^{(1)}=\del\left\{\frac{1}{2}\Tr(\bm\ta\bm\Theta\bm\tab-\bm\tab\bm\Theta\bm\ta)\right\} +\frac{1}{2}\Tr(\bm\theta^{3}\bm\tab^{2}-\bm\ta\bm\tab\bm\ta\bm\Theta).$$
We also have
 
\begin{align*}
\del\Tr(\bm\ta\bm\tab)^{2} &=2\Tr\!\left((-\bm\ta^{2}\bm\tab-\bm\ta\del\bm\tab)\bm\ta\bm\tab)\right)=2\Tr\!\left((-\bm\ta^{2}\bm\tab+\bm\ta\bm\Theta+\bm\ta\bm\ta\bm\tab+\bm\ta\bm\tab\bm\ta)\bm\ta\bm\tab)\right)\\
& =2\Tr\!\left(\bm\ta\bm\Theta\bm\ta\bm\tab+\bm\ta\bm\tab\bm\ta\bm\ta\bm\tab\right),
\end{align*}
so that
$$\Tr(\bm\ta\bm\tab\bm\ta\bm\Theta) =\del\left\{\frac{1}{2}\Tr(\bm\ta\bm\tab)^{2}\right\}-\Tr(\bm\ta^{2}\bm\tab\bm\ta\bm\tab).$$
Thus we obtain
$$\omega_{3,2}+\delb\omega_{3,1}^{(1)}=\del\left\{\frac{1}{2}\Tr\!\left(\bm\ta\bm\Theta\bm\tab-\bm\tab\bm\Theta\bm\ta-\frac{1}{2}(\bm\ta\bm\tab)^{2}\right)\right\} +\frac{1}{2}\Tr\!\left(\bm\theta^{3}\bm\tab^{2}+\bm\ta^{2}\bm\tab\bm\ta\bm\tab\right).$$

So far we have not used the Cholesky decomposition and now we start using it for this case. Note that
$$\omega_{3,2}+\delb\omega_{3,1}^{(1)}=\del\left\{\frac{1}{2}\Tr\!\left(\bm\ta\bm\Theta\bm\tab-\bm\tab\bm\Theta\bm\ta-\frac{1}{2}(\bm\ta\bm\tab)^{2}\right)\right\} +\frac{1}{2}\Tr\!\left(\theta^{3}\tab^{2}+\ta^{2}\tab\ta\tab\right)$$
and it remains to compute
\begin{align*}
\delb\omega_{3,1}^{(2)} & =\frac{1}{2}\delb\Tr\!\left(\bm I_{1}\bm\tab_{1} +\bm I_{2}\bm\tab_{2}\right)=\frac{1}{2}\delb\Tr\!\left(I_{1}\tab_{1} +I_{2}\tab_{2}\right) \\
& =\frac{1}{2}\Tr\!\left(\delb I_{1}\tab_{1} +\delb I_{2}\tab_{2} - I_{1}\tab_{1}^{2} + I_{2}\tab_{2}^{2}\right).
\end{align*}
By a straightforward computation using 
$$\delb\ta=-\del\tab +\ta_{1}\tab_{1}+\tab_{1}\ta_{1}-\ta_{2}\tab_{2}-\tab_{2}\ta_{2}$$
we get
\begin{gather*}
\delb I_{1}\tab_{1} = \\
=\Tr\!\left\{\left[\delb\ta\ta^{2}-\ta\delb\ta\ta+\ta^{2}\delb\ta-\delb\ta\ta_{2}\ta_{1} +\ta(\delb\ta_{2}\ta_{1}-\ta_{2}\delb\ta_{1})-\right.\right.\\
\left.\left. -(\delb\ta_{1}\ta_{2}-\ta_{1}\delb\ta_{2})\ta-\ta_{1}\ta_{2}\delb\ta\right]\tab_{1}\right\} \\
 =\Tr\!\left\{\!\left[-\del\tab\ta^{2}+\ta\del\tab\ta-\ta^{2}\del\tab +\del\tab\ta_{2}\ta_{1}+\ta(\ta_{2}\del\tab_{1}-\del\tab_{2}\ta_{1})+\right.\right.\\
\left.\left. +(\del\tab_{1}\ta_{2}-\ta_{1}\del\tab_{2})\ta+\ta_{1}\ta_{2}\del\tab\,\right]\tab_{1}\right\}+\\
+\Tr\!\left\{\!\left[(\ta_{1}\tab_{1}+\tab_{1}\ta_{1}-\ta_{2}\tab_{2}-\tab_{2}\ta_{2})\ta^{2}-\ta(\ta_{1}\tab_{1}+\tab_{1}\ta_{1}-\ta_{2}\tab_{2}-\tab_{2}\ta_{2})\ta+\right.\right.\\
+\ta^{2}(\ta_{1}\tab_{1}+\tab_{1}\ta_{1}-\ta_{2}\tab_{2}-\tab_{2}\ta_{2})-(\ta_{1}\tab_{1}+\tab_{1}\ta_{1}-\ta_{2}\tab_{2}-\tab_{2}\ta_{2})\ta_{2}\ta_{1}-\\
-\ta((\ta_{2}\tab_{2}+\tab_{2}\ta_{2})\ta_{1}+\ta_{2}(\ta_{1}\tab_{1}+\tab_{1}\ta_{1}))-((\ta_{1}\tab_{1}+\tab_{1}\ta_{1})\ta_{2}+\ta_{1}(\ta_{2}\tab_{2}+\tab_{2}\ta_{2}))\ta -\\
\left.\left.-\ta_{1}\ta_{2}(\ta_{1}\tab_{1}+\tab_{1}\ta_{1}-\ta_{2}\tab_{2}-\tab_{2}\ta_{2})\right]\tab_{1}\right\}
\end{gather*}
and
\begin{gather*}
\delb I_{2}\tab_{2}  = \\
=\Tr\!\left\{\left[-\delb\ta\ta^{2}+\ta\delb\ta\ta-\ta^{2}\delb\ta+\delb\ta\ta_{1}\ta_{2} -\ta(\delb\ta_{1}\ta_{2}-\ta_{1}\delb\ta_{2})+\right.\right.\\
\left.\left. +(\delb\ta_{2}\ta_{1}-\ta_{2}\delb\ta_{1})\ta+\ta_{2}\ta_{1}\delb\ta\right]\tab_{2}\right\} \\
 =\Tr\!\left\{\!\left[\del\tab\ta^{2}-\ta\del\tab\ta+\ta^{2}\del\tab -\del\tab\ta_{1}\ta_{2}-\ta(\ta_{1}\del\tab_{2}-\del\tab_{1}\ta_{2})-\right.\right.\\\left.\left. -(\del\tab_{2}\ta_{1}-\ta_{2}\del\tab_{1})\ta-\ta_{2}\ta_{1}\del\tab\,\right]\tab_{2}\right\}+\\
+\Tr\!\left\{\!\left[-(\ta_{1}\tab_{1}+\tab_{1}\ta_{1}-\ta_{2}\tab_{2}-\tab_{2}\ta_{2})\ta^{2}+\ta(\ta_{1}\tab_{1}+\tab_{1}\ta_{1}-\ta_{2}\tab_{2}-\tab_{2}\ta_{2})\ta-\right.\right.\\
 -\ta^{2}(\ta_{1}\tab_{1}+\tab_{1}\ta_{1}-\ta_{2}\tab_{2}-\tab_{2}\ta_{2})+(\ta_{1}\tab_{1}+\tab_{1}\ta_{1}-\ta_{2}\tab_{2}-\tab_{2}\ta_{2})\ta_{1}\ta_{2}-\\
-\ta(\ta_{1}(\ta_{2}\tab_{2}+\tab_{2}\ta_{2})+(\ta_{1}\tab_{1}+\tab_{1}\ta_{1})\ta_{2})-(\ta_{2}(\ta_{1}\tab_{1}+\tab_{1}\ta_{1})+(\ta_{2}\tab_{2}+\tab_{2}\ta_{2})\ta_{1})\ta +\\
\left.\left.+\ta_{2}\ta_{1}(\ta_{1}\tab_{1}+\tab_{1}\ta_{1}-\ta_{2}\tab_{2}-\tab_{2}\ta_{2})\right]\tab_{2}\right\}.
\end{gather*}
Thus we obtain
$$\Tr(\delb I_{1}\tab_{1}+\delb I_{2}\tab_{2})=J_{1}+J_{2},$$
where
\begin{align*}
J_{1} & =\Tr\!\left\{-\del\tab_{2}(\ta_{1}^{2}+\ta_{2}^{2}+\ta_{1}\ta_{2})\tab_{1}+(\ta_{1}^{2}+\ta_{2}^{2}+\ta_{1}\ta_{2})\del\ta_{1}\tab_{2}-(\ta_{1}^{2}+\ta_{2}^{2}+\ta_{2}\ta_{1})\del\tab_{2}\tab_{1}\right.\\
&\quad +\del\tab_{1}(\ta_{1}^{2}+\ta_{2}^{2}+\ta_{2}\ta_{1})\tab_{2}+\del\tab_{1}(\ta_{2}\ta_{1}-\ta_{1}\ta_{2})\tab_{1}-(\ta_{2}\ta_{1}-\ta_{1}\ta_{2})\del\tab_{1}\tab_{1}\\
&\quad+\del\tab_{2}(\ta_{2}\ta_{1}-\ta_{1}\ta_{2})\tab_{2}-(\ta_{2}\ta_{1}-\ta_{1}\ta_{2})\del\tab_{2}\tab_{2}+\ta_{2}\del\tab_{1}\ta_{2}\tab_{1}+\ta_{2}\del\tab_{2}\ta_{2}\tab_{1}\\
&\quad-\ta_{1}\del\tab_{1}\ta_{1}\tab_{2}-\ta_{1}\del\tab_{2}\ta_{1}\tab_{2}+\ta_{1}\del\tab_{1}\ta_{2}\tab_{1}+\ta_{2}\del\tab_{1}\ta_{1}\tab_{1}\\
&\quad\left.-\ta_{1}\del\tab_{2}\ta_{2}\tab_{2}-\ta_{2}\del\tab_{2}\ta_{1}\tab_{2}-\ta_{1}\del\tab_{2}\ta_{1}\tab_{1}+\ta_{2}\del\tab_{1}\ta_{2}\tab_{2}\right\}
\end{align*}
and
\begin{align*}
J_{2} &=\Tr\!\left\{\left[(\ta_{1}\tab_{1}+\tab_{1}\ta_{1}-\ta_{2}\tab_{2}-\tab_{2}\ta_{2})\ta^{2}-\ta(\ta_{1}\tab_{1}+\tab_{1}\ta_{1}-\ta_{2}\tab_{2}-\tab_{2}\ta_{2})\ta\right.\right.\\
&\quad\left.+\ta^{2}(\ta_{1}\tab_{1}+\tab_{1}\ta_{1}-\ta_{2}\tab_{2}-\tab_{2}\ta_{2})\right](\tab_{1}-\tab_{2})-(\ta_{1}\tab_{1}+\tab_{1}\ta_{1}-\ta_{2}\tab_{2}-\tab_{2}\ta_{2})(\ta_{2}\ta_{1}\tab_{1}-\ta_{1}\ta_{2}\tab_{2}) \\
&\quad -\ta_{1}\ta_{2}(\ta_{1}\tab_{1}+\tab_{1}\ta_{1}-\ta_{2}\tab_{2}-\tab_{2}\ta_{2})\tab_{1}+\ta_{2}\ta_{1}(\ta_{1}\tab_{1}+\tab_{1}\ta_{1}-\ta_{2}\tab_{2}-\tab_{2}\ta_{2})\tab_{2} \\
&\quad -\ta((\ta_{2}\tab_{2}+\tab_{2}\ta_{2})\ta_{1}+\ta_{2}(\ta_{1}\tab_{1}+\tab_{1}\ta_{1}))\tab_{1}-\ta(\ta_{1}(\ta_{2}\tab_{2}+\tab_{2}\ta_{2})+(\ta_{1}\tab_{1}+\tab_{1}\ta_{1})\ta_{2})\tab_{2}\\
&\left.\quad-(\ta_{1}(\ta_{2}\tab_{2}+\tab_{2}\ta_{2})+(\ta_{1}\tab_{1}+\tab_{1}\ta_{1})\ta_{2})\ta\tab_{1}-((\ta_{2}\tab_{2}+\tab_{2}\ta_{2})\ta_{1}+\ta_{2}(\ta_{1}\tab_{1}+\tab_{1}\ta_{1})\ta\tab_{2}\right\}.
\end{align*}
Simplifying and using the cyclic property of the trace, we get
\begin{align*}
J_{1}&=\Tr\!\left\{(\ta_{1}^{2}+\ta_{2}^{2}+\ta_{1}\ta_{2})\del(\tab_{1}\tab_{2})-(\ta_{1}^{2}+\ta_{2}^{2}+\ta_{2}\ta_{1})\del(\tab_{2}\tab_{1})-(\ta_{2}\ta_{1}-\ta_{1}\ta_{2})\del(\tab_{1}^{2}+\tab_{2}^{2})\right.\\
&\quad+(\ta_{2}\del\tab_{1}\ta_{2}\tab_{2}+\ta_{2}\tab_{1}\ta_{2}\del\tab_{2})-(\ta_{1}\del\tab_{1}\ta_{1}\tab_{2}+\ta_{1}\tab_{1}\ta_{1}\del\tab_{2})+(\ta_{2}\del\tab_{1}\ta_{1}\tab_{1}+\ta_{2}\tab_{1}\ta_{1}\del\tab_{1})\\
&\left.\quad-(\ta_{2}\del\tab_{2}\ta_{1}\tab_{2}+\ta_{2}\tab_{2}\ta_{1}\del\tab_{2})+\ta_{2}\del\tab_{1}\ta_{2}\tab_{1}-\ta_{1}\del\tab_{2}\ta_{1}\tab_{2}\right\}\\
&=\del\left\{\Tr\!\left((\ta_{1}^{2}+\ta_{2}^{2}+\ta_{1}\ta_{2})\tab_{1}\tab_{2}-(\ta_{1}^{2}+\ta_{2}^{2}+\ta_{2}\ta_{1})\tab_{2}\tab_{1}-(\ta_{2}\ta_{1}-\ta_{1}\ta_{2})(\tab_{1}^{2}+\tab_{2}^{2})\right.\right.\\
&\left.\left.\quad-\ta_{2}\tab_{1}\ta_{2}\tab_{2}+\ta_{1}\tab_{1}\ta_{1}\tab_{2}-\ta_{2}\tab_{1}\ta_{1}\tab_{1}+\ta_{2}\tab_{2}\ta_{1}\tab_{2}+\tfrac{1}{2}((\ta_{1}\tab_{2})^{2}-(\ta_{2}\tab_{1})^{2}\right)\right\}\\
&\quad+\Tr\left\{-(\ta_{1}^{2}\ta_{2}+\ta_{1}\ta_{2}^{2})\tab_{1}\tab_{2}-(\ta_{2}^{2}\ta_{1}+\ta_{2}\ta_{1}^{2})\tab_{2}\tab_{1}-(\ta_{1}^{2}\ta_{2}+\ta_{1}\ta_{2}^{2}+\ta_{2}^{2}\ta_{1}+\ta_{2}\ta_{1}^{2})(\tab_{1}^{2}+\tab_{2}^{2})\right.\\
&\quad -\ta_{2}^{2}\tab_{1}\ta_{2}\tab_{2}-\ta_{2}^{2}\tab_{2}\ta_{2}\tab_{1}-\ta_{1}^{2}\tab_{1}\ta_{1}\tab_{2}-\ta_{1}^{2}\tab_{2}\ta_{1}\tab_{1}-\ta_{2}^{2}\tab_{1}\ta_{1}\tab_{1}+\ta_{1}^{2}\tab_{1}\ta_{2}\tab_{1}\\
&\quad\left.+\ta_{2}^{2}\tab_{2}\ta_{1}\tab_{2}-\ta_{1}^{2}\tab_{2}\ta_{2}\tab_{2}-\ta_{1}^{2}\tab_{2}\ta_{1}\tab_{2}-\ta_{2}^{2}\tab_{1}\ta_{2}\tab_{1}\right\}=J_{11}+J_{12},
\end{align*}
where
\begin{align*}
J_{11}& = \del\left\{\Tr\!\left((\ta_{1}^{2}+\ta_{2}^{2}+\ta_{1}\ta_{2})\tab_{1}\tab_{2}-(\ta_{1}^{2}+\ta_{2}^{2}+\ta_{2}\ta_{1})\tab_{2}\tab_{1}-(\ta_{2}\ta_{1}-\ta_{1}\ta_{2})(\tab_{1}^{2}+\tab_{2}^{2})\right.\right.\\
&\left.\left.\quad-\ta_{2}\tab_{1}\ta_{2}\tab_{2}+\ta_{1}\tab_{1}\ta_{1}\tab_{2}-\ta_{2}\tab_{1}\ta_{1}\tab_{1}+\ta_{2}\tab_{2}\ta_{1}\tab_{2}+\tfrac{1}{2}((\ta_{1}\tab_{2})^{2}-(\ta_{2}\tab_{1})^{2}\right)\right\}\\
\intertext{and}
J_{12}&= \Tr\left\{-(\ta_{1}^{2}\ta_{2}+\ta_{1}\ta_{2}^{2})\tab_{1}\tab_{2}-(\ta_{2}^{2}\ta_{1}+\ta_{2}\ta_{1}^{2})\tab_{2}\tab_{1}-(\ta_{1}^{2}\ta_{2}+\ta_{1}\ta_{2}^{2}+\ta_{2}^{2}\ta_{1}+\ta_{2}\ta_{1}^{2})(\tab_{1}^{2}+\tab_{2}^{2})\right.\\
&\quad -\ta_{2}^{2}\tab_{1}\ta_{2}\tab_{2}-\ta_{2}^{2}\tab_{2}\ta_{2}\tab_{1}-\ta_{1}^{2}\tab_{1}\ta_{1}\tab_{2}-\ta_{1}^{2}\tab_{2}\ta_{1}\tab_{1}-\ta_{2}^{2}\tab_{1}\ta_{1}\tab_{1}+\ta_{1}^{2}\tab_{1}\ta_{2}\tab_{1}\\
&\quad\left.+\ta_{2}^{2}\tab_{2}\ta_{1}\tab_{2}-\ta_{1}^{2}\tab_{2}\ta_{2}\tab_{2}-\ta_{1}^{2}\tab_{2}\ta_{1}\tab_{2}-\ta_{2}^{2}\tab_{1}\ta_{2}\tab_{1}\right\},\\
\end{align*}
and
\begin{align*}
J_{2} & = \Tr\!\left\{\!\left(\ta^{2}(\tab_{1}-\tab_{2})(\ta_{1}\tab_{1}-\ta_{2}\tab_{2}) + \ta_{1}\ta^{2}(\tab_{1}-\tab_{2})\tab_{1} -\ta_{2}\ta^{2}(\tab_{1}-\tab_{2})\tab_{2}\right.\right.\\
&\quad -\ta(\ta_{1}\tab_{1}-\ta_{2}\tab_{2})\ta(\tab_{1}-\tab_{2})-\ta_{1}\ta(\tab_{1}-\tab_{2})\ta\tab_{1}+\ta_{2}\ta(\tab_{1}-\tab_{2})\ta\tab_{2}\\
&\quad +\ta^{2}(\ta_{1}\tab_{1}-\ta_{2}\tab_{2})(\tab_{1}-\tab_{2})+\ta^{2}(\tab_{1}\ta_{1}-\tab_{2}\ta_{2})(\tab_{1}-\tab_{2})-\ta_{2}\ta_{1}\tab_{1}\ta_{1}\tab_{1} \\
&\quad- \ta_{1}\ta_{2}\ta_{1}\tab_{1}^{2} 
 + \ta_{2}\ta_{1}\tab_{1}\ta_{2}\tab_{2} + \ta_{2}\ta_{2}\ta_{1}\tab_{1}\tab_{2} +\ta_{1}\ta_{2}\tab_{2}\ta_{1}\tab_{1}+\ta_{1}^{2}\ta_{2}\tab_{2}\tab_{1}\\ 
 &\quad-\ta_{1}\ta_{2}\tab_{2}\ta_{2}\tab_{2} -\ta_{2}\ta_{1}\ta_{2}\tab_{2}^{2}
-\ta_{1}\ta_{2}(\ta_{1}\tab_{1}-\ta_{2}\tab_{2})\tab_{1}-\ta_{1}\ta_{2}(\tab_{1}\ta_{1}-\tab_{2}\ta_{2})\tab_{1}\\
&\quad + \ta_{2}\ta_{1}(\ta_{1}\tab_{1}-\ta_{2}\tab_{2})\tab_{2}+\ta_{2}\ta_{1}(\tab_{1}\ta_{1}-\tab_{2}\ta_{2})\tab_{2}
-\ta\ta_{2}\tab\ta_{1}\tab_{1}-2\ta_{2}\ta_{1}\tab_{1}\ta\tab_{2}\\
&\quad-\ta\ta_{2}\ta_{1}\tab^{2}_{1}-\ta\ta_{1}\tab\ta_{2}\tab_{2}-\ta\ta_{1}\ta_{2}\tab_{2}^{2}-2\ta_{1}\ta_{2}\tab_{2}\ta\tab_{1}
 -\ta_{2}\ta\tab_{1}\ta_{1}\tab_{2}
-\ta_{2}\ta\tab_{1}\ta_{1}\tab_{1}\\
&\quad -\ta_{1}\ta_{2}\ta \tab_{1}^{2}
\left.-\ta_{1}\ta\tab_{2}\ta_{2}\tab_{2}-\ta_{2}\ta_{1}\ta\tab_{2}^{2}-\ta_{1}\ta\tab_{2}\ta_{2}\tab_{1}\right\}\!.
\end{align*} 
Simplifying $J_{12}+J_2$ once again and after using numerous `miraculous cancellations', we obtain
\begin{align*}
\tilde{J}_{1}+J_{2}+\Tr(-I_{1}\tab_{1}^{2}+I_{2}\tab_{2}^{2})) 
&=-\Tr\!\left(\theta^{3}\tab^{2}+\ta^{2}\tab\ta\tab\right),
\end{align*}
so that finally
\begin{align*}
\omega_{3,2}+\delb\omega_{3,1} & =\del\Big\{\frac{1}{2}\Tr\!\big(\bm\ta\bm\Theta\bm\tab-\bm\tab\bm\Theta\bm\ta-\frac{1}{2}(\ta\tab)^{2}+(\ta_{1}^{2}+\ta_{2}^{2}+\ta_{1}\ta_{2})\tab_{1}\tab_{2}\\
&\quad -(\ta_{1}^{2}+\ta_{2}^{2}+\ta_{2}\ta_{1})\tab_{2}\tab_{1}- (\ta_{2}\ta_{1}-\ta_{1}\ta_{2})(\tab_{1}^{2}+\tab_{2}^{2})-\ta_{2}\tab_{1}\ta_{2}\tab_{2}\\
&\quad +\ta_{1}\tab_{1}\ta_{1}\tab_{2}-\ta_{2}\tab_{1}\ta_{1}\tab_{1}+\ta_{2}\tab_{2}\ta_{1}\tab_{2}+\frac{1}{2}((\ta_{1}\tab_{2})^{2}-(\ta_{2}\tab_{1})^{2}\big)\Big\}. 
\end{align*}
Thus we obtain the following result.
\begin{proposition} \label{3} The third Bott-Chern form $\mathrm{bc}_{3}$ of a trivial Hermitian vector bundle $(\CC^{r},h)$ over a complex manifold $X$ in Cholesky coordinates $h=b^{*}ab$ is given by the formula
\begin{align*}
\mathrm{bc}_{3} & =\frac{1}{12}\Tr\!\left(\bm\ta\bm\Theta\bm\tab-\bm\tab\bm\Theta\bm\ta-\frac{1}{2}(\ta\tab)^{2}+(\ta_{1}^{2}+\ta_{2}^{2}+\ta_{1}\ta_{2})\tab_{1}\tab_{2}\right.\\
&\quad -(\ta_{1}^{2}+\ta_{2}^{2}+\ta_{2}\ta_{1})\tab_{2}\tab_{1}- (\ta_{2}\ta_{1}-\ta_{1}\ta_{2})(\tab_{1}^{2}+\tab_{2}^{2})-\ta_{2}\tab_{1}\ta_{2}\tab_{2}\\
&\left.\quad +\ta_{1}\tab_{1}\ta_{1}\tab_{2}-\ta_{2}\tab_{1}\ta_{1}\tab_{1}+\ta_{2}\tab_{2}\ta_{1}\tab_{2}+\frac{1}{2}((\ta_{1}\tab_{2})^{2}-(\ta_{2}\tab_{1})^{2}\right).
\end{align*}
\end{proposition}


\begin{thebibliography}{99}
\bibitem{GH} Phillip Griffiths and Joseph Harris, \emph{Principles of algebraic geometry}, Wiley, 1978.
\bibitem{CS} S.S. Chern and J. Simons, \emph{Characteristic forms and geometric invariants}, Ann. of Math. (2) \textbf{99},  (1974), 48--69.
\bibitem{BC} R. Bott and S.S. Chern, \emph{Hermitian vector bundles and the equidistribution of the zeroes of their holomorphic sections}, Acta Math. \textbf{114} (1965), 71--112.
\bibitem{D} S. Donaldson, \emph{Anti-self dual Yang-Mills connections over complex surfaces and stable vector bundles}, Proc. London Math. Soc., \textbf{50} (1985), 1-26.
\bibitem{Tian} G. Tian, \emph{Bott-Chern forms and geometric stability}, Discrete Contin. Dynam. Systems, \textbf{6} (2000), 211-220.
\bibitem{GS1} H. Gillet and C. Soul\'{e}, \emph{Characteristic classes for algebraic vector bundles with hermitian metric, I}, Ann. of Math., \textbf{131} (1990), 163-203.
\bibitem{Tam} H. Tamvakis, \emph{Arithmetic intersection theory on flag varieties}, Math. Ann. \textbf{314} (1999), 641-665.
\bibitem{LMNS} Andrei Losev, Gregory Moore, Nikita Nekrasov, Samson Shatashvili, \emph{Chiral Lagrangians, Anomalies, Supersymmetry, and Holomorphy}, Nucl.Phys. \textbf{B484} (1997), 196-222.
\bibitem{PT} Vamsi P. Pingali and Leon A. Takhtajan, \emph{On Bott-Chern forms and their applications}, Math. Ann. \textbf{360}:1-2 (2014), 519-546.  
\bibitem{FS} L.D. Faddeev and S.L. Shatashvili, \emph{Algebraic and Hamiltonian methods in the theory of non-Abelian anomalies}, Theoret. and Math. Phys., \textbf{60}:2 (1984), 770-778.
\bibitem{SS} James H. Simons and Dennis Sullivan, \emph{Structured vector bundles define differential
  {$K$}-theory}, In: Quanta of maths, \emph{Clay Math. Proc.}, \textbf{11}, pp. 579--599; Amer. Math. Soc., Providence, RI (2010).
\bibitem{CV} Claire Voisin, \emph{Hodge theory and complex algebraic geometry II}, Cambridge U. Press, 2003.
\bibitem{Igusa} K. Igusa, \emph{Higher Franz-Reidemeister torsion}, AMS/IP Studies in Advanced Mathematics, \textbf{31}, 2002.
\end{thebibliography}
\end{document}